\documentclass[letterpaper, 11pt]{amsart}
\usepackage{amssymb,amsmath,amsfonts,amsthm}
\usepackage{graphicx}
\usepackage[normalem]{ulem}
\graphicspath{ {./images/} }
\usepackage{setspace}

\usepackage{mathrsfs}  
\usepackage{psfrag}
\usepackage{mathtools}
\usepackage{color}
\usepackage{todonotes}
\usepackage{enumitem}

\definecolor{Red}{rgb}{1,0,0}
\definecolor{Blue}{rgb}{0,0,1}
\definecolor{Green}{rgb}{0,1,0}
\definecolor{magenta}{rgb}{1,0,.6}
\definecolor{gold}{rgb}{.6,.5,0}
\definecolor{orange}{rgb}{1,0.4,0}
\definecolor{darkgreen1}{rgb}{0, .35, 0}
\definecolor{darkgreen}{rgb}{0, .6, 0}
\definecolor{darkred}{rgb}{.75,0,0}

\usepackage{hyperref}
\hypersetup{colorlinks=true, citecolor=green, linkcolor=blue, hypertexnames=false}

\usepackage{chngcntr}
\counterwithin{equation}{section}

\theoremstyle{plain}
\newtheorem{main}{Theorem}

\newtheorem{theorem}{Theorem}[section]
\newtheorem{lemma}[theorem]{Lemma}
\newtheorem{proposition}[theorem]{Proposition}

\theoremstyle{remark}
\newtheorem{remark}[theorem]{Remark}
\newtheorem{definition}{Definition}

\newcommand\numberthis{\addtocounter{equation}{1}\tag{\theequation}}

               \def\cal{\mathcal}
           \def\ea{\end{array}}
          \def\ec{\end{center}}
     \def\ed{\end{description}}
        \def\ee{\end{equation}}
       \def\eea{\end{eqnarray}}
     \def\eeaa{\end{eqnarray*}}
 \def\et{\end{thebibliography}}

\def\bM{{\bf{M}}}

\def\U{{\cal U}}

\def\Cl{{\rm Cl}}

\newcommand{\interior}[1]{%
	{\kern0pt#1}^{\mathrm{o}}%
}

\def\supp{\operatorname{supp}}

\def\cG{{\mathcal G}}

\def\cD{{\mathcal D}}
\def\cC{{\mathcal C}}
\def\cO{{\mathcal O}}

\def\cM{{\mathcal M}}
\def\cN{{\mathcal N}}
\def\cP{{\mathcal P}}

\def\cx{{\mathbf  x}}

\def\cS{{\mathcal S}}

\def\vep{\varepsilon}

\def\RR{{\mathbb R}}

\def\NN{{\mathbb N}}

\def\Exp{\operatorname{Exp}}
\def\NE{{\rm NE}}

	\title[An Improved CT-Criterion for Locally Maximal Sets]{An Improved Climenhaga-Thompson Criterion for Locally Maximal Sets}
\date{\today}
\author{Maria Jose Pacifico, Fan Yang, Jiagang Yang and Gongran Yao}

\address{M.J. Pacifico, Instituto de Matem\'atica, Universidade Federal do Rio de Janeiro, C. P. 68.530, CEP 21.945-970,  Rio de Janeiro, RJ, Brazil.}
 \email{pacifico@im.ufrj.br }

\address{F. Yang, Department of Mathematics, Wake Forest University, Winston-Salem, NC, USA.}
\email{yangf@wfu.edu; fizbanyang@gmail.com}

\address{J. Yang, Departamento de Geometria, Instituto de Matem\'atica e Estat\'\i stica, Universidade Federal Fluminense, Niter\'oi, Brazil}
\email{yangjg\@@impa.br}

\address{G. Yao, Department of Mathematics, Wake Forest University, Winston-Salem, NC, USA.}
\email{yaog23@wfu.edu }

\thanks{MJP and JY were partially supported by CAPES-Finance Code 001, CNPq-Projeto Universal No. 404943/2023-3, and CNPq-Brazil No. 307776/2019-0 and 302975/2019-5 respectively. MJP was partially supported by FAPERJ Grant- CNE E-26/202.850/2018(239069). JY was partially supported by NSFC12271538, NSFC11871487, NSFC 12071202 and MATH-AmSud 220029. FY and GY were partially supported by NSF Award Number 2418590.}          

\begin{document}

\begin{abstract}
We study the existence and uniqueness of equilibrium states for  continuous flows on a compact, locally maximal invariant set under weak, non-uniform versions of specification, expansivity, and the Bowen property, further improving the Climenhaga-Thompson Criterion \cite{CT16, PYY21}. 
\end{abstract}

\maketitle

\section{Introduction}
The Climenhaga-Thompson Criterion (abbrev.\ CT Criterion) \cite{CT16} establishes the existence of a unique equilibrium state when certain weak, non-uniform versions of the specification, expansivity, and the Bowen property hold at some fixed scales. It has been widely used to establish the uniqueness of equilibrium states for geodesic flows (\cite{BCFT, CKW}), diffeomorphisms (\cite{CFT18, CFT19}), as well as certain symbolic systems (\cite{GPR}).  

The authors of \cite{CT16} assume that either (1) specification holds at a fixed scale on a family of orbit segments that is much larger than the ``good core''; or (2) specification holds on the ``good core'' at arbitrarily small scales (in fact (2) implies (1)). More recently,  in \cite{PYY21} we relaxed the specification assumption, asking it to hold only on the ``good core'' at a fixed scale  and achieved the same result. This improvement allowed us to apply the improved Climenhaga-Thompson Criterion to flows with singularities, proving that every sectional-hyperbolic attractor, including the classical Lorenz attractor, must support a unique equilibrium states \cite{PYY23}.

One of the issues that we encountered in \cite{PYY23} is that each sectional-hyperbolic attractor $\Lambda$ is a compact, locally maximal invariant set. As a result, the shadowing orbit given by the specification may not be in $\Lambda$ any more. In \cite{PYY23} we were able to bypass this issue by showing that for sectional-hyperbolic attractors,  it is always possible to select the shadowing orbit from $\Lambda$, exploiting the fact that $\Lambda$ is an attractor and therefore contains the unstable manifold of every periodic orbit in it. 

In an on-going project \cite{PYY25}, we aim to prove that every star vector field supports only finitely many equilibrium states.\footnote{This paragraph is aimed to provide some background information; as a result, we will not provide the precise definitions, and invite interested readers to \cite{PYY23a} and \cite{PYY23}.} This requires us to consider locally maximal invariant sets that are not attractors. In particularly, there are two cases: (1) the so-called multi-singular hyperbolic sets, which are of saddle type; and (2) sectional-hyperbolic sets that are not attractors. For this purpose, below we will make a further improvement for the CT Criterion in \cite{PYY21}, which can be applied to locally maximal invariant sets. For simplicity, we will only prove a version for continuous flows. The same result holds for homeomorphisms following the treatment in \cite[Section 5]{CT16}.
 
\section{Preliminary and Statement of Result}\label{sec2}
Let $(f_t)_{t\in\RR}$ be a continuous flow on a compact metric space $\bM$, and $\Lambda$ be a compact, locally maximal invariant set of $(f_t)_{t\in\RR}$ with an isolating neighborhood $U$; that is,
\begin{equation}\label{e.iso}
\Lambda = \bigcap_{t\in\RR} f_t(\Cl(U))
\end{equation} 
where $\Cl$ denotes the closure of a set.
Note that every neighborhood of $\Lambda$ contained in $U$ is also an isolating neighborhood of $\Lambda$. In particular, we take $U_1$ an open neighborhood of $\Lambda$ satisfying 
$$
\Lambda\subset U_1 \subset \Cl(U_1)\subset U.
$$   
More properties of $U_1$ will be specified later. 


We consider the following types of orbit segments:
\begin{itemize}
	\item for a set $I\subset \RR$ we write $f_{I}(x) = \{f_sx: s\in I\}$. 
	\item Every element $(x,t)\in \Lambda\times\RR^+$ is identified with the orbit segment $f_{[0,t)}(x) = \{f_sx: s\in [0,t)\}$; these are finite orbit segments that are contained in $\Lambda$. As in \cite{CT16} and \cite{PYY21}, we consider $(x,0)$ as the empty set rather than the singleton $\{x\}$. 
	\item We write $\cO(U)$ for the subset of $U\times \RR$ such that every $(x,t)\in \cO(U)$ satisfies
	$$
	f_sx\in U, \forall s\in [0,t);
	$$
	then we identify $(x,t)\in \cO(U)$ with the orbit segment $f_{[0,t)}(x)\subset U$. In other words, $\cO(U)$ consists of finite orbit segments entirely contained in $U$.
	\item $\cO(U_1)$ is defined similarly. 
\end{itemize}
From the definition, it follows that 
$$
\Lambda\times \RR\subset \cO(U_1) \subset \cO(U).
$$
In this paper, due to our setting of $U$ and $U_1$, we will not consider orbit segments that are not contained in $\cO(U)$.

Next, we consider the pressure on a collection of finite orbit segments. We define the Bowen metric 
$$d_t(x, y) = \sup\{d(f_sx, f_sy): s \in [0, t]\},$$ for given $t \geq 0$ and $x, y \in \bM$;  we also define 
Bowen ball $B_t(x, \delta)$ for $\delta > 0$ as $$B_t(x, \delta) = \{y \in \bM : d_t(x, y) < \delta\}$$ for $t,\delta \geq 0.$  
For an orbit segment collection $\cC$ 
we define 
$$
(\cC)_t = \{x\in\bM: (x,t)\in \cC\}. 
$$
We say a set $E_t \subset  (\cC)_t \subset \bM$ is a $(t,\delta)$-separated set of $\cC$ if every distinct pair $x, y \in E_t$ satisfies $d_t(x, y) > \delta$. 

Let $\phi: U\to \RR$ be a continuous function, and  $\delta>0,\vep>0$ be two positive (small) constants. The {\em (two-scale) partition function} $\lambda(\cC,\phi,\delta,\vep,t)$\footnote{In \cite{CT16, PYY21} the partition function is denoted by $\Lambda$; here we use $\lambda$ to avoid confusion with the invariant set $\Lambda$.} for a collection of finite orbit segments $\cC$ and $t>0$ is defined as 
\begin{equation}\label{e.Lambda}
	\lambda(\cC,\phi,\delta,\vep,t)=\sup\left\{\sum_{x\in E}e^{\Phi_\vep(x,t)}:E\subset (\cC)_t \mbox{ is $(t,\delta)$-separated}\right\},
\end{equation}
where 
$$
\Phi_\vep(x,t) = \sup_{y\in B_t(x,\vep)} \left(\int_0^t \phi(f_s(x))\,ds\right); 
$$
so $\Phi_0(x,t)$ is the usual Birkhoff integral on $(x,t).$
Any $(t,\delta)$-separated set that attains the supremum of $\lambda(\cC, \phi, \delta, \varepsilon, t)$ is called maximizing. Note that $(\cC)_t$ may not be compact in general, and so maximizing separated sets may not exist. 

To simplify notation, below we will frequently drop the potential $\phi$ and write $\lambda(\cC, \delta, \varepsilon, t)$ instead.

%

\begin{remark}
	There are several types of monotonicity of  $\lambda(\cC,\phi,\delta,\vep,t)$: 
	(1)  $\lambda(\cC,\phi,\delta,\vep,t)$ is monotonically decreasing as $\vep$ decreases, and (2) monotonically increasing as $\delta$ decreases; (3) when there are two collections of orbit segments $\cC^1 \subset \cC^2,$ we have $\lambda(\cC^1,\phi,\delta,\vep,t) \le \lambda(\cC^2,\phi,\delta,\vep,t).$ 
\end{remark}

The pressure of $\phi$ on $\cC$ with scale $\delta,\vep$ is defined as
\begin{equation}\label{e.P1}
	P(\cC,\phi,\delta,\vep) = \limsup_{t\to+\infty}\frac1t \log\lambda(\cC,\phi,\delta,\vep,t).
\end{equation}
The monotonicity of the partition function can be naturally translated to $P$. When $\vep=0$ we will often write $P(\cC,\phi,\delta)$, and let
\begin{equation}\label{e.P2}
	P(\cC,\phi) = \lim_{\delta\to0} P(\cC,\phi,\delta). 
\end{equation}
When $\cC =  \Lambda\times \RR^+$, this coincides with the standard definition of the topological pressure $P(\phi;\Lambda)$ for the restriction of the flow on $\Lambda$.

\subsection{ Obstruction to expansivity and almost expansivity}\label{s.2.2}
As in \cite{CT16,PYY21} we consider the set 
$$
\Gamma_\vep(x)=\{y\in \bM: d(f_tx,f_ty)\le \vep \mbox{ for all }t\in\RR\},
$$
which will be called a (two-sided) infinite Bowen ball at $x$. It is easy to see that $\Gamma_\vep(x)$ is compact for all $x$ and $\vep$.

\begin{definition}\label{d.EXP}
	Given $\vep>0$, the set of expansive points in $\Lambda$ at scale $\vep$ is the set 
	\begin{equation}\label{e.Exp}
		\Exp(\vep;\Lambda) :=  \left\{x\in  \Lambda: \Gamma_\vep(x)\subset f_{[-s,s]}(x) \mbox{ for some }s>0\right\}.
	\end{equation}
	The complement of $\Exp(\vep;\Lambda)$ in $\Lambda$, which we denote by $\NE(\vep; \Lambda)$, is called the set of non-expansive points at scale $\vep$. 
		
	We sat that an $f$-invariant probability measure $\mu$ supported in $\Lambda$ is almost expansive on $\Lambda$ at scale $\vep$, if $\mu(\NE(\vep;\Lambda))=0$. We also say the flow is almost expansive on $\Lambda$ at sale $\vep$ if every $f$-invariant probability measure is almost expansive on $\Lambda$ at scale $\vep$.
\end{definition}

For the next definition, we write $\cM_f^e(\Lambda)$ for the set of invariant ergodic probability measures whose support is contained in $\Lambda$. 
\begin{definition}\label{d.OP}
	Given a potential $\phi: U\to \RR$, the pressure of obstruction to expansivity on $\Lambda$ at scale $\vep$ is defined as 
	$$
	P_{\exp}^\perp (\phi,\vep;\Lambda) = \sup_{\mu\in\cM_f^e(\Lambda)}\left\{h_\mu(f_1)+\int \phi\,d\mu: \mu(\operatorname{NE}(\vep;\Lambda))=1\right\}.
	$$
\end{definition}
	
Where the flow is almost expansive on $\Lambda$, the supremum is taken over an empty set; following the standard notation that  $\sup\emptyset = -\infty$, we see that if $f_t$ is almost expansive on $\Lambda$ at scale $\vep$ then $P_{\exp}^\perp (\phi,\vep';\Lambda) = -\infty$ for all $\vep'\le\vep$.
	
	
We conclude this subsection by remarking that if $P_{\exp}^\perp (\phi,\vep;\Lambda)<P(\phi;\Lambda)$, then any equilibrium state $\mu$ supported in $\Lambda$ must be almost expansive at scale $\vep$. 
	

\subsection{Decomposition of orbit segments}\label{s.2.3}
\begin{definition}\cite[Definition 2.3]{CT16}
	A decomposition  $(\cP,\cG,\cS)$ for a collection of finite orbit segments $\cD$ consists of three collections $\cP,\cG,\cS$ and three functions $p,g,s:\cD\to \RR^+$ such that for every $(x,t)\in \cD$, the values $p=p(x,t), g=g(x,t)$ and $s=s(x,t)$ satisfy $t=p+g+s$, and
	\begin{equation}\label{e.decomp}
		(x,p)\in\cP,\hspace{0.5cm} (f_px,g)\in\cG,\hspace{0.5cm} (f_{p+g}x,s)\in\cS.
	\end{equation}
	Given a decomposition $(\cP,\cG,\cS)$ and real number $M\ge0$, we write $\cG^M$ for the set of orbit segments $(x,t)\in \cD$ with $p\le  M$ and $s\le  M$.
\end{definition}

As mentioned earlier, in this article, we must deal with three ``spaces'' of finite orbit segments: $\Lambda\times\RR^+\subset \cO(U_1)\subset \cO(U)$. Suppose there exists\footnote{This notation should not be confused with $(\cC)_t$ defined earlier. Instead we will write $(D_1)_1$ for the set $\{x\in\bM: (x,1)\in D_1\}$.} $\cD_1\subset \cO(U_1)$ with a $(\cP_{1},\cG_1,\cS_1)-$decomposition on (here we use the subscript to highlight the fact that the decomposition is defined for orbit segments in $\cO(U_1)$); then the following properties hold:
\begin{enumerate}
	\item $\cP_1,\cG_1,\cS_1$ are subsets of $\cO(U_1)$; consequently, orbit segments in them are contained in $U_1$;
	\item by restricting the functions $p,g,s$ to $\cD_0:= \cD_1\cap\Lambda\times\RR^+$ ( which {\em a priori} may be empty) and considering $\cP_0 := \cP_1\cap\Lambda\times\RR^+$ (similarly define $\cG_0,\cS_0$), one obtains a  $(\cP_0,\cG_0,\cS_0)-$decomposition of $\cD_0\subset \Lambda\times\RR^+$. 
\end{enumerate}

\subsection{The Bowen property and specification}

Let $\cC$ be a collection of finite orbit segments. Below we will give the definitions of the Bowen property and specification on $\cC.$
\begin{definition}
    We say a potential $\phi$ has the Bowen property on $\cC$ at scale $\varepsilon$ if there exists $K>0$ so that
\begin{equation*}
\sup \left\{\left|\Phi_{0}(x, t)-\Phi_{0}(y, t)\right|:(x, t) \in \cC, y \in B_{t}(x, \varepsilon)\right\} \le K.
\end{equation*}
\end{definition}

$K$ is sometimes called the distortion constant for the Bowen property. Note that if $\phi$ has the Bowen property at scale $\varepsilon$ on $\cG_1$ with distortion constant $K$, then $\phi$ has the Bowen property on any subset of $\cG_1$ at any scale that is less than $\vep$ with the same distortion constant $K$.
Furthermore, $\phi$ has the Bowen property at scale $\varepsilon$ on $(\cG_1)^{M}$ with distortion constant given by $K(M)=K+2 M \operatorname{Var}(\phi, \varepsilon)$ for any $M>0$, where 
$$
\operatorname{Var}(\phi,\vep) = \sup_{d(x,y)<\vep} |\phi(x)-\phi(y)|.
$$


\begin{definition}\label{d.spec11}	Let $\cO(U_1)$, $\cO(U)$ be defined as before. 
	\begin{enumerate}
		\item Let $\cC_0\subset \Lambda\times\RR^+. $ We say that $\cC_0$ has tail (W)-specification at scale $\delta$ with shadowing orbit contained in $\cO(U_1)$,  if there exist $\tau>0$ and $T_0 > 0$ such that for every $\left\{\left(x_{i}, t_{i}\right)\right\}_{i=1}^{k} \subset \cC_0$ with $t_i\geq T_0 $, there exist a point $y\in U_1$ and a sequence of "gluing times" $\tau_{1}, \ldots, \tau_{k-1} \in \mathbb{R}^{+}$ with $\tau_{i} \le \tau$ such that for $s_{j}=\sum_{i=1}^{j} t_{i}+\sum_{i=1}^{j-1} \tau_{i}$ and $s_{0}=\tau_{0}=0$, we have
		$$
		d_{t_{j}}\left(f_{s_{j-1}+\tau_{j-1}} y, x_{j}\right)<\delta \text { for every } 1 \le j \le k;
		$$
		Furthermore, we have
		$(y,s_k)\in \cO(U_1).
		$
		\item Let $\cC_1\subset\cO(U_1). $ We say that $\cC_1$ has tail (W)-specification at scale $\delta$ with shadowing orbit contained in $\cO(U)$,  if there exist $\tau>0$ and $T_0 > 0$ such that for every $\left\{\left(x_{i}, t_{i}\right)\right\}_{i=1}^{k} \subset \cC_1$ with $t_i\geq T_0 $, there exist a point $y\in U$ and a sequence of "gluing times" $\tau_{1}, \ldots, \tau_{k-1} \in \mathbb{R}^{+}$ with $\tau_{i} \le \tau$ such that for $s_{j}=\sum_{i=1}^{j} t_{i}+\sum_{i=1}^{j-1} \tau_{i}$ and $s_{0}=\tau_{0}=0$, we have
		$$
		d_{t_{j}}\left(f_{s_{j-1}+\tau_{j-1}} y, x_{j}\right)<\delta \text { for every } 1 \le j \le k;
		$$
		furthermore,
		$
		(y,s_k)\in \cO(U).
		$
	\end{enumerate}
	Note that both cases require orbit segments $(f_{s_{j}}y, \tau_j)$ that correspond to the transition from $(x_j,t_j)$ to $(x_{j+1}, t_{j+1})$ to be contained in a larger neighborhood, namely $U_1$ and $U$ respectively. Of course, this would require $\delta>0$ to be sufficiently small. 
\end{definition}

\subsection{The Main Result}

For a collection of orbit segments  $\cC$ we define the slightly larger collection $[\cC]\supset \cC$ to be
\begin{equation}\label{e.[C]}
	[\cC]:= \{(x,n)\in  \bM\times \mathbb{N}:(f_{-s}x,n+s+t) \in\cC \mbox{ for some } s,t\in[0,1)\}.
\end{equation}
This allows us to pass from continuous time to discrete time.

Our main result is:
\begin{main}\label{m.1}
	Let $(f_t)_{t\in\RR}$ be a continuous flow on a compact metric space $X$. Assume that $\Lambda$ is a compact invariant set of $(f_t)$ that is locally maximal with isolating neighborhood $U$, and $\phi: U\to\RR$ a continuous potential function. Suppose that there exist $\vep>0,\delta>0$ with $\vep\ge1000 \delta$, such that  $P_{\exp}^\perp (\phi,\vep;\Lambda)<P(\phi;\Lambda),$ 
 together with an open neighborhoods $U_1$ of $\Lambda$ satisfying 
	$$
	\Lambda\subset U_1\subset U.
	$$
	 Also assume there exists $\cD_1\subset  \cO(U_1)$ which admits a  $(\cP_1,\cG_1,\cS_1)-$decomposition that induces a $(\cP_0,\cG_0,\cS_0)-$decomposition on $\cD_0 = \cD_1\cap \Lambda\times\RR^+$ with the following properties: 
	\begin{enumerate}[label={(\Roman*)}]
		\hypertarget{I0}{\item [($I_0$)]} $(\cG_0)^1$ has tail (W)-specification at scale $\delta$ with shadowing orbit contained in $\cO(U_1)$.
		\hypertarget{I1}{\item [($I_1$)]} $(\cG_1)^1$ has tail (W)-specification at scale $\delta$ with shadowing orbit contained in $\cO(U)$.
		\item[(II)]  $\phi$ has the Bowen property at scale $\vep$ on $\cG_1$.
		\item[($III$)] $P((\cO(U_1)\setminus\cD_1)\cup [\cP_1]\cup[\cS_1], \phi,\delta,\vep)<P(\phi;\Lambda)$. 
	\end{enumerate}
	Then, there exists a unique equilibrium state for the potential $\phi|_\Lambda$ whose support is contained in $\Lambda$ and is ergodic.
\end{main}

\begin{remark}\label{r.gap}
	Condition $(III)$ in Theorem~\ref{m.1} implies the following ``pressure gap'' property on $\Lambda$: 
		$$
		P((\Lambda\times\RR^+\setminus\cD_0)\cup [\cP_0]\cup[\cS_0], \phi,\delta,\vep)<P(\phi;\Lambda).
		$$
		This is because 
		$$
		[\cP_0]\subset [\cP_1],\,\, [\cS_0]\subset [\cS_1], \text{ and } \Lambda\times\RR^+\setminus\cD_0\subset (\cO(U_1)\setminus\cD_1).
		$$
		In particular, this shows that $\cD_0\ne\emptyset$ and satisfies $P(\cD_0, \phi) =  P(\phi;\Lambda).$

\subsection{Strategy of proof}
The key ingredient to obtain the uniqueness of equilibrium state is the lower Gibbs property on $\cG_0$, namely  Lemma \ref{l.gibbs}. In the proof, the specification on  $\Lambda$ is used to create shadowing orbits in $\cO(U_1)$. As a result, we must obtain:
\begin{enumerate}
	\item a lower bound for the partition sum on $\cG_0$; and 
	\item an upper bound for the partition sum on $\cO(U_1)$.
\end{enumerate} 
The proof of the former requires the pressure gap on $\Lambda$ and is relatively similar to \cite{CT16,PYY21}. For the latter, one must obtain an upper bound on $\cG_1$,  then make use of the pressure gap on $\cO(U_1)$; this shows that all orbit segments collections in $\cO(U_1)$ with large enough pressure must substantially intersect with $\cG_1$; from there, the upper bound on $\cG_1$ translates to an upper bound on $\cO(U_1)$. In this step, the upper bound on $\cG_1$ calls for the specification on $\cG_1$, and the shadowing orbits are contained in $\cO(U)$. As a starting point, we must show that $\Lambda,\cO(\U_1),\cO(U)$ all have the same pressure; this is done in Lemma \ref{l.equalpressure} below.



\end{remark}


\section{Lower bound on $\Lambda\times\RR^+$ and upper bound on $\cO(U_1)$}\label{sec3}

In this section, we estimate the growth rate of partition sums on $\Lambda\times\RR^+$ and $\cO(U_1).$ Note that due to monotonicity of the partition function, we usually only need to prove the lower bound on $\Lambda\times\RR^+$ and the upper bound on $\cO(U_1).$

To begin with, we show $\cO(U),\cO(U_1)$ and $\Lambda\times\RR^+$ have equal pressure.

\begin{lemma}\label{l.equalpressure}
	Let $\Lambda$ be a locally maximal compact invariant set with isolating neighborhood $U$. Let $U_1$ be any neighborhood of $\Lambda$ contained in $U$, and $\phi: U\to\RR$ a continuous function. Then we have 
	$$
	P(\cO(U),\phi) = P(\cO(U_1),\phi) = P(\Lambda\times\RR^+, \phi) = P(\phi;\Lambda). 
	$$
\end{lemma}
\begin{proof}
	By monotonicity of the partition sum, we always have
	$$P(\Lambda\times\RR^+, \phi) \le  P(\cO(U_1),\phi)  \le P(\cO(U),\phi)  $$
	Therefore it suffices to show that 	
	$$
	P(\cO(U),\phi)  \le P(\Lambda\times\RR^+, \phi).
	$$
	
	To this end, we let $\gamma>0$ be arbitrary. For any $t>0$, we let $E_t$ be a $(t, \gamma)$-separated set of $\cO(U)_t$ satisfying
   $$\log\sum_{x \in E_{T}} e^{\Phi_{0}(x, T)} > \log\lambda( \cO(U_1), \gamma, T) - 1.$$
    Then define the measures
$$v_t = \frac{\sum_{x \in E_t} e^{\Phi_0(x, t)}\cdot \delta_x}{\sum_{x \in E_t} e^{\Phi_0(x, t)}},$$  and
$$u_t = \frac{1}{t} \int_0^t (f_s)_{\ast}v_t ds.$$
By the compactness of the space of probability measures on $\bM$, there exists a subsequence  $\mu_{t_k}$ converging to an invariant probability measure $\mu$ (possibly depending on $\gamma$) in the $weak^\ast$ topology. Furthermore, $\supp\mu\subset \Cl(U)$ is an invariant set. Since $\Lambda$ is locally maximal with $U$ being an isolating neighborhood (see \eqref{e.iso}), we must have $\supp\mu\subset\Lambda.$

Now due to our choice of $E_t$, the proof of the variational principle \cite[Lemma 8.6]{walters2000} shows that $$P(\cO(U),\phi, \gamma) \le P_\mu(\cO(U),\phi).$$ Again by the variational principle, we have 
    $$P_\mu(\cO(U),\phi)\le P(\Lambda\times\RR^+, \phi),$$
    and finally by the arbitrarity of $\gamma$ we obtain $$P(\cO(U),\phi) \le P(\Lambda\times\RR^+, \phi),$$ as required.
   
\end{proof}

For each $\delta>0$ we define the closed interval
\begin{equation}\label{e.I}
	I_\delta = [8 \delta,\,\, 200 \delta].
\end{equation}
Throughout this article, the parameter $\gamma$ will be taken from $I_\delta$. 


Below we will state a few preliminary lemmas; their proof will not require any specification, and therefore can be obtained by applying the corresponding results in \cite{CT16, PYY21} to $\Lambda = X$.
\begin{lemma}\label{l.4.1}\cite[Lemma 4.1]{CT16} For every $\gamma>0$ and $t_1,\ldots, t_k>0$ we have 
	$$
	\lambda(\Lambda\times\RR^+, 2\gamma,t_1+\cdots+ t_k)\le \prod_{j=1}^k \lambda( \Lambda\times \RR^+,\gamma,\gamma,t_j).
	$$
\end{lemma}
Taking logarithm and sending $k$ to infinity, we obtain:
\begin{lemma}(See \label{l.4.2}\cite[Lemma 4.2]{CT16} and \cite[Lemma 5.2]{PYY21}) Assume that $P_{\exp}^\perp (\phi,\vep;\Lambda)<P(\phi;\Lambda)$ where $\vep=1000 \delta$. For every $t>0$ and $\gamma\in I_\delta$, we have 
	$$
	\lambda(\Lambda\times\RR^+, \gamma,\gamma,t)\ge e^{tP(\phi;\Lambda)}.
	$$
\end{lemma}

Since $\Lambda\times\RR^+\subset\cO(U_1)\subset \cO(U)$, we immediately obtain the following lemma from monotonicity. 

\begin{lemma}\label{l.4.2aa} Assume that $P_{\exp}^\perp (\phi,\vep;\Lambda)<P(\phi;\Lambda)$ where $\vep=1000 \delta$. For every $t>0$ and $\gamma\in I_\delta$, we have 
	$$
	\lambda(\cO(U), \gamma,\gamma,t)\ge e^{tP(\phi;\Lambda)}.
	$$
	Same holds for $\cO(U_1).$
\end{lemma}

In the next lemma, we will use specification on $(\cG_1)^1$ to obtain a lower bound on the growth of partition sum on $(\cG_1)^1$. Due to Assumption (\hyperlink{I1}{$I_1$}), the shadowing orbit will be taken from $\cO(U)$.


\begin{lemma}\label{l.4.3}(See \cite[Proposition 4.3]{CT16} and \cite[Lemma 5.3]{PYY21}) 
		Suppose that Assumption (\hyperlink{I1}{$I_1$}) holds for $t>T_0$ with maximum gap size $\tau$, and the potential $\phi$ has the Bowen property on $\cG_1$ at scale $\vep$. Then for every $\gamma\in I_\delta$ and every $\theta<\gamma/2-\delta$, there is a constant $C_1>0$ so that for every $k\in\NN$, $t_1,\ldots, t_k\ge T_0$, writing $T:=\sum_{i=1}^kt_i + (k-1)\tau$, we have 
		$$
		\prod_{j=1}^{k} \lambda((\cG_1)^1,\gamma,t_j)\le C_1^k\lambda(  \cO(U),\theta,T).
		$$

\end{lemma}
\begin{proof}
	
	{ We fix any $(z,t_0)\in (\cG_1)^1$ with $t_0>T_0$. For a given $k\in\NN$ we fix $N$ the smallest positive integer such that $N \cdot t_0\ge k\tau$.  }

     
    Since $(\cO(U))_{t_j}$ and $((\cG_1)^1)_{t_j}$ may not be compact, we let $E_{j} \subset ((\cG_1)^1)_{t_j}$ be $(t_{j}, \gamma)$-separated sets and 
     $E_T \subset (\cO(U))_T$ be $(\theta, T)$-separated sets such that $$\sum_{x \in E_{j}}e^{\Phi_{0}(x, t_j)} > \lambda((\cG_1)^1,\gamma, t_j) -  \epsilon, \mbox{ and } \sum_{x \in E_{T}} e^{\Phi_{0}(x, T)}> \lambda( \cO(U_1), \theta, T) -  \epsilon$$ for a small $\epsilon > 0$. 
    
    \begin{figure}[h!]
    	\centering
    	\def\svgwidth{\columnwidth}
    	\includegraphics[scale=1.05]{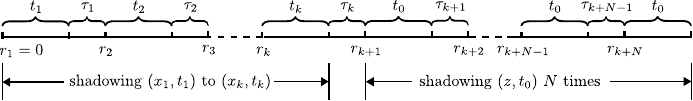}
    	\caption{Applying the specification}
    	\label{f.1}
    \end{figure}
    
     By Assumption (\hyperlink{I1}{$I_1$}), for any vector $\mathbf{x} = (x_{1}, \ldots, x_{k}) \in \prod_{j=1}^{k} E_{j}$,  there are $y=y(\mathbf{x}) \in U$ and $\boldsymbol{\tau}=\boldsymbol{\tau}(\mathbf{x})=(\tau_{1}, \ldots, \tau_{k-1}, {\tau_k,\cdots, \tau_{k+N-1}  }) \in[0, \tau]^{k+N-1}$ such that, with $t_i = t_0$ for $i=k+1,\ldots, k+N$ and $r_1(\mathbf{x})=0,$
     $$r_j(\mathbf{x}) =\sum_{i=1}^{j-1}(t_{i}+\tau_{i}),$$
	one has	
      $f_{r_i}(y) \in B_{t_{i}}(x_{i}, \delta)$ for every $i=1,\cdots, k$ and $f_{r_i}(y) \in B_{t_{0}}(z)$ for $i = k+1,\cdots, k+N$ (see Figure \ref{f.1}). 
	Furthermore, the shadowing orbits $(y,r_{k+N}(\mathbf{x}) + t_{0})$ is contained in $\cO(U)$. As a result, we have 
	$$
	y\in (\cO(U))_{r_{k+N}(\mathbf{x}) + t_{0}}.
	$$ 
	From the definition of $\cO(U)$ it is clear that 
	$$
	(\cO(U))_{s} \subset (\cO(U))_{t} \mbox{ whenever } s\ge t.
	$$
	Since $r_{k+N}(\mathbf{x}) + t_{0}\ge T$ due to the choice of $N$, we obtain 
	$$
	y\in (\cO(U))_{T}.
	$$

	Since $E_T $ is a maximizing $(\theta, T)$-separated sets of $(\cO(U))_T$, it must be $(\theta, T)$-spanning for $(\cO(U))_T$. As a result, we can find a point $x \in E_T$ closest to $y$ in $d_T$-metric. Now we can define $\pi: \prod E_{j} \rightarrow E$, $\pi$ maps each $(x_{1}, \ldots, x_{k}) \in \prod_{j=1}^{k} E_{j}$ to a point in $E_T$ as described above.   We now need to control the multiplicity of $\pi$. 

     Now for $\mathbf{x}, \mathbf{x}' \in \prod_{j=1}^{k} E_{j},$ for simplicity we let $r_j(\mathbf{x}) = r_j, r_j(\mathbf{x}') = r_j'$ for any $j$. 
     By triangle inequality, we have
     \begin{equation}\label{e.tri}
          d_{t_{j}}(x_{j}, f_{r_j}(\pi \mathbf{x})) \le \delta+\theta \qquad d_{t_{j}}(x_{j}^{\prime}, f_{r_j^{\prime}}(\pi \mathbf{x}^{\prime})) \le \delta+\theta
     \end{equation}
     and, for every $j=1,\cdots, k$, 
      \begin{align}\label{e.4.3}
         d_{T}(\pi \mathbf{x}, \pi \mathbf{x}^{\prime}) \geq\,& d_{t_{j}}(f_{r_{j}} \pi \mathbf{x}, f_{r_{j}} \pi \mathbf{x}^{\prime}) \notag\\
         \geq \,& d_{t_{j}}(f_{r_{j}} \pi \mathbf{x}, f_{r_{j}^{\prime}} \pi \mathbf{x}^{\prime})-d_{t_{j}}(f_{r_{j}} \pi \mathbf{x}^{\prime}, f_{r_{j}^{\prime}} \pi \mathbf{x}^{\prime}) \notag\\
         \geq \,& d_{t_{j}}(x_{j}, x_{j}^{\prime})-d_{t_{j}}(x_{j}, f_{r_{j}} \pi \mathbf{x})-d_{t_{j}}(f_{r_{j}^{\prime}} \pi \mathbf{x}^{\prime}, x_{j}^{\prime}) \\ &-  d_{t_{j}}(f_{r_{j}} \pi \mathbf{x}^{\prime}, f_{r_{j}^{\prime}} \pi \mathbf{x}^{\prime}).  \notag
     \end{align}  
     
     Take $m$ large enough such that $d(x, f_{s} x)<\gamma - 2(\delta + \theta)$ for every $x \in \bM$ and every $s \in(-\tau / m, \tau / m)$. 
    We partition the interval $[0, k \tau]$ into $k m$ sub-intervals $I_{1}, \ldots, I_{k m}$ of length $\tau / m$, denoting this partition $P$. 
    Given $\mathbf{x} \in \prod_j E_{j}$, take the sequence $n_{1}, \ldots, n_{k}$ so that
$$
\tau_{1}(\mathbf{x})+\cdots+\tau_{i}(\mathbf{x}) \in I_{n_{i}} \text { for every } 1 \le i \le k-1.
$$
Now let $\ell_{1}=n_{1}$ and $\ell_{i+1}=n_{i+1}-n_{i}$ for $1 \le i \le k-2$, and let $\ell(\mathbf{x}):=(\ell_{1}, \ldots, \ell_{k-1})$. Since $\tau_{i+1}(\mathbf{x}) \in[0, \tau]$, we have $n_{i} \le n_{i+1} \le n_{i}+m$ for each $i$, and thus $\ell(\mathbf{x}) \in\{0, \ldots, m-1\}^{k-1}$.

Now given $\bar{\ell} \in\{0, \ldots, m-1\}^{k-1}$, let $E^{\bar{\ell}} \subset \prod E_{j}$ be the set of all $\mathbf{x}$ such that ${\ell}(\mathbf{x})=\bar{\ell}$. Note that if $\mathbf{x}, \mathbf{x}^{\prime} \in E^{\bar{\ell}}$ and $i \in\{1, \ldots, k-1\}$, then by construction, $\tau_{1}(\mathbf{x})+\cdots+\tau_{i}(\mathbf{x})$ and $\tau_{1}^{\prime}(\mathbf{x})+\cdots+\tau_{i}^{\prime}(\mathbf{x})$ belong to the same element of the partition $P$.

We show that $\pi$ is $1-1$ on each $E^{\bar{\ell}}$. Fix $\overline{{\ell}}$ and let $\mathbf{x}, \mathbf{x}^{\prime} \in E^{\bar{\ell}}$ be distinct. Let $j$ be the smallest index such that $x_{j} \neq x_{j}^{\prime}$, we have $|r_i-r_i^{\prime}|=|\tau_{i}- \tau_{i}^{\prime}|<\tau / m$ if $i < j$. Using (\ref{e.4.3}) we get:
\begin{align*}
         d_{T}(\pi \mathbf{x}, \pi \mathbf{x}^{\prime}) \geq  \,& d_{t_{j}}(x_{j}, x_{j}^{\prime})-d_{t_{j}}(x_{j}, f_{r_{j}} \pi \mathbf{x})-d_{t_{j}}(f_{r_{j}^{\prime}} \pi \mathbf{x}^{\prime}, x_{j}^{\prime}) \\ &-  d_{t_{j}}(f_{r_{j}} \pi \mathbf{x}^{\prime}, f_{r_{j}^{\prime}} \pi \mathbf{x}^{\prime})  \\
         >\, & \gamma - 2(\delta + \theta) - d_{t_{j}}(f_{r_{j}} \pi \mathbf{x}^{\prime}, f_{r_{j}^{\prime}} \pi \mathbf{x}^{\prime})\\ >\, & 0,
     \end{align*}  
   where the last inequality is due to the choice of $m$.
   This proves that there is a constant $C$ depends on $\gamma, \delta, \theta $  such that $\# \pi^{-1}(x) \le C^{k}$ for every $x \in E_T$, and a simple calculation (here $\|\phi\|$ denotes the $C^0$ norm of $\phi$)
    gives
    \begin{align*}
        \Phi_{0}(\pi \mathbf{x}, T)=\int_{0}^{T} \phi\left(f_{s}(\pi \mathbf{x})\right) d s \geq & -(k-1) \tau\|\phi\|+\sum_{j=1}^{k} \Phi_{0}\left(f_{r_j}(\pi \mathbf{x}), t_{j}\right)\\  \geq &  -(k-1) \tau\|\phi\|+ \sum_{j=1}^{k}(\Phi_{0}\left(x_{j}, t_{j}\right)-K) \\ \geq & -k(\tau\|\phi\|+K)+\sum_{j=1}^{k} \Phi_{0}\left(x_{j}, t_{j}\right).
    \end{align*}
   Here $K$ is the distortion constant that appears in the Bowen property. We obtain  
\begin{align*}
\sum_{z \in E_T} e^{\Phi_{0}(z, T)} \geq & C^{-k} \sum_{\mathbf{x} \in \prod_{j} E_{j}} e^{\Phi_{0}(\pi \mathbf{x}, T)} \\
 \geq & C^{-k} e^{-k (\tau\|\phi\|+K)} \sum_{\mathbf{x} \in \prod_{j} E_{j}} \prod_{j} e^{\Phi_{0}\left(x_{j}, t_{j}\right)}.
\end{align*}
 We take $\epsilon \rightarrow 0$ to get the inequality we need. 
\end{proof}
We get the following lemma using the monotonicity of $ \lambda((\cG_1)^1, \gamma, \eta_1, t) $ in both scales and the Bowen property.
\begin{lemma}(See \cite[Corollary 4.6]{CT16} and \cite[Lemma 5.4]{PYY21})\label{l.4.6} 
	Under the assumptions of Lemma \ref{l.4.3}, for every $\eta_1 \le\vep$ and for every $\eta_2>0$ one has 
	$$\prod_{j=1}^{k}\lambda((\cG_1)^1,\gamma,\eta_1,t_j)\le (e^KC_1)^k\lambda( \cO(U), \theta,\eta_2,T).$$
	Here $K$ is the constant in the Bowen property.	
\end{lemma}

Taking logarithm, sending $k$ to infinity and applying Lemma \ref{l.equalpressure}, we obtain the following upper bound on $(\cG_1)^1$. 

\begin{proposition}(See \cite[Proposition 4.7]{CT16} and \cite[Proposition 5.5, Lemma 6.1]{PYY21})\label{p.4.7}
	Suppose that $(\cG_1)^1$ has tail specification at scale $\delta>0$ {with shadowing orbit contained in $U$}, and $\phi$ has the Bowen property on $\cG_1$ at scale $\vep$. Then for every $\gamma\in I_\delta$ and $\eta\in[0,\vep]$, there exists a constant $C_2>0$ for which
	$$\lambda((\cG_1)^1, \gamma,\eta,t)\le C_2 e^{tP(\phi;\Lambda)}$$
	for any $t\ge 0$. Since $(\cG_0)^1\subset (\cG_1)^1$, we also have 
	$$\lambda((\cG_0)^1, \gamma,\eta,t)\le C_2 e^{tP(\phi;\Lambda)}.$$
\end{proposition}

In the next two lemmas, we establish some lower bound for the partition sums on $\cG_1$ and $\cG_0$.

\begin{lemma}{(See \cite[Lemma 4.8]{CT16} and \cite[Lemma 5.7]{PYY21})}\label{l.4.8}
	Let $(\cP_1,\cG_1,\cS_1)$ be a decomposition for $\cD_1\subset \cO(U_1)$ and $(\cP_0,\cG_0,\cS_0)$ be the induced decomposition for $\cD_0\subset \Lambda\times \RR^+$, 
	such that
	\begin{enumerate}
		\item $(\cG_1)^1$ has (tail) specification at scale $\delta$  with shadowing orbit contained in $U$;
		\item the potential $\phi$ has the Bowen property on $\cG_1$ at scale $\vep$; and
		\item $P((\cO(U_1)\setminus\cD_1)\cup [\cP_1]\cup[\cS_1], \phi,\delta,\vep)<P(\phi;\Lambda)$.
	\end{enumerate}
	Then for every $\alpha_1,\alpha_2>0$ and $\gamma\in I_\delta$, there exist $M=M(\gamma,\alpha_1,\alpha_2)\in\RR^+$ and $T_1=T_1(\gamma,\alpha_1,\alpha_2)\in\RR^+$ such that the following holds:
	\begin{itemize}
		\item for any $t\ge T_1$ and $\cC\subset  \Lambda\times\RR^+$ with $\lambda(\cC,2\gamma,2\gamma,t)\ge\alpha_1e^{tP(\phi;\Lambda)}$, we have
		\begin{equation}\label{e.4.8}
			\lambda(\cC\cap(\cG_0)^M,2\gamma,2\gamma,t)\ge(1-\alpha_2)\lambda(\cC,2\gamma,2\gamma,t)\ge c_1e^{tP(\phi;\Lambda)},
		\end{equation}
		where $c_1 = \alpha_1(1-\alpha_2)$.
		
	\end{itemize}
\end{lemma}

The proof is omitted. We make the following observation: 
\begin{enumerate}
	\item $(\cG_1)^1$ having (tail) specification at scale $\delta$  with shadowing orbit contained in $U$, together with the Bowen property on $\cG_1$, imply that  (Proposition \ref{p.4.7}) 
	$$
	\lambda((\cG_0)^1, \gamma,\eta,t)\le C_2 e^{tP(\phi;\Lambda)}.
	$$
	\item $P((\cO(U_1)\setminus\cD_1)\cup [\cP_1]\cup[\cS_1], \phi,\delta,\vep)<P(\phi;\Lambda)$ implies that (Remark \ref{r.gap})
	$$
	P((\Lambda\times\RR^+\setminus\cD_0)\cup [\cP_0]\cup[\cS_0], \phi,\delta,\vep)<P(\phi;\Lambda).
	$$
\end{enumerate}
Then the proof follows verbatim the proof of \cite[Lemma 4.8]{CT16}.

Following the same argument, we get a version of Lemma \ref{l.4.8} for the partition sum of $\cG_1$. Here monotonicity is not enough because this time we have $\cC\subset \cO(U_1)$ instead of $\Lambda\times\RR^+.$
\begin{lemma}\label{l.4.8'}
	Let $(\cP_1,\cG_1,\cS_1)$ be a decomposition for $\cD_1\subset \cO(U_1)$ 
	such that
	\begin{enumerate}
		\item $ (\cG_1)^1$ has (tail) specification at scale $\delta$   with shadowing orbit contained in $U$;
		\item the potential $\phi$ has the Bowen property on $ \cG_1$ at scale $\vep$; and
		\item $P(\cO(U_1)\setminus\cD_1)\cup [\cP_1]\cup[\cS_1], \phi,\delta,\vep)<P(\phi;\Lambda)$.
	\end{enumerate}
	Then for every $\alpha_1,\alpha_2>0$ and $\gamma\in I_\delta$, there exist $M=M(\gamma,\alpha_1,\alpha_2)\in\RR^+$ and $T_1=T_1(\gamma,\alpha_1,\alpha_2)\in\RR^+$ such that the following holds:
	\begin{itemize}
		\item for any $t\ge T_1$ and $ \cC\subset  \cO(U_1)$ with $\lambda(\cC,2\gamma,2\gamma,t)\ge\alpha_1e^{tP(\phi;\Lambda)}$, we have
		\begin{equation}\label{e.4.8'}
			\lambda(\cC\cap(\cG_1)^M,2\gamma,2\gamma,t)\ge(1-\alpha_2)\lambda(\cC,2\gamma,2\gamma,t)\ge c_1e^{tP(\phi;\Lambda)},
		\end{equation}
		where $c_1 = \alpha_1(1-\alpha_2)$.
	\end{itemize}
\end{lemma}
Here we assume W.L.O.G that the constants $M,T_1$ given in Lemma \ref{l.4.8} and \ref{l.4.8'} are the same. There is no restriction to assume that $T_1 > 2M$. We will do so for the rest of this paper.

The following lemmas are direct consequences of Lemma \ref{l.4.8} and \ref{l.4.8'}. 

\begin{lemma}\label{l.4.9}(See \cite[Lemma 4.9]{CT16} and \cite[Lemma 5.9]{PYY21})
	Assume that $(\cG_1)^1$ has tail specification at scale $\delta$  with shadowing orbit contained in $U$, $P_{\exp}^\perp (\phi,\vep)<P(\phi;\Lambda)$, and $\phi$ has the Bowen property  on $\cG_1$ at scale $\vep$. Then for every $\gamma \in I_\delta$ and $\alpha_2>0$, the constants $M\in\NN$ and $T_1\in\RR^+$ given by Lemma \ref{l.4.8} with $\alpha_1=1$ satisfies that for every $t\ge T_1$,
	$$
	\lambda((\cG_0)^M, 2\gamma,2\gamma,t)\ge (1-\alpha_2)\lambda( \Lambda\times\RR^+,2\gamma,2\gamma,t)\ge (1-\alpha_2)e^{tP(\phi;\Lambda)}.
	$$
\end{lemma}
\begin{proof}
    Apply Lemma \ref{l.4.8} with $\cC = \Lambda\times\RR^+$ and  $\alpha_1 = 1$. 
\end{proof}
\begin{lemma}\label{l.4.10}(See \cite[Proposition 4.10]{CT16} and \cite[Lemma 5.10]{PYY21})
	Let $\vep,\delta,\gamma$ be as in Lemma \ref{l.4.9}, and $M,T_1$ be given by Lemma \ref{l.4.8} with $\alpha_1=1$ and $\alpha_2=\frac12$. Then there exists $L_1 = L_1(\gamma)\in\RR^+$ such that for $t\ge T_1$, 
	$$
	\lambda((\cG_0)^M,2\gamma,t)\ge e^{-L_1}e^{tP(\phi;\Lambda)}.
	$$
	As a consequence, $\lambda(\Lambda\times\RR^+,2\gamma,t)\ge e^{-L_1}e^{tP(\phi;\Lambda)} $ for $t\ge T_1$.
	
	 Since $\cO(U_1)\supset \Lambda\times\RR^+$ we get
	$$
	\lambda(\cO(U_1),2\gamma,t)\ge e^{-L_1}e^{tP(\phi;\Lambda)}.
	$$
\end{lemma}
\begin{proof}
    Using the previous lemma and the Bowen property on $(\cG_0)^M$ with distortion constant $K(M) = K+2 M \operatorname{Var}(\phi, \varepsilon),$ we have 
\begin{align*}
    \lambda((\cG_0)^M,2\gamma,t)\ge &\,e^{-K(M)}\lambda((\cG_0)^M,2\gamma,2\gamma,t) \\
    \ge&\, \frac{1}{2}e^{-K(M)}e^{tP(\phi;\Lambda)} \\
    = &\, e^{-L_1}e^{tP(\phi;\Lambda)}.
\end{align*} 
\end{proof}

Below, we try to reproduce all the estimates in \cite[Section 6]{PYY21}. One key difference between \cite{CT16} and \cite{PYY21} is that the latter does not assume the specification on $\cG^M$. Therefore, we need to enhance Lemma \ref{l.4.8} and \ref{l.4.8'}.

As in Lemma \ref{l.4.10}, we take $\alpha_2 = \frac12$ and write 
\begin{equation}\label{e.MT}
	M(\gamma, \alpha_1)=M(\gamma, \alpha_1,\frac12) \mbox{ and } T_1 (\gamma, \alpha_1)= T_1(\gamma, \alpha_1,\frac12).
\end{equation} 
Given $\cC\subset \cO(U)$ and $i,j\in\NN$, we write
\begin{equation}\label{e.fC}
	f_{i,j}(\cC) = \{(f_i(x),t-(i+j)):(x,t)\in\cC,  t\ge i+j\}.
\end{equation}
Note that if $\cC\subset \cO(U_1)$ then $f_{i,j}(\cC)\subset \cO(U_1)$. Similarly, if $\cC\subset \Lambda\times\RR^+$ then $f_{i,j}(\cC)\subset\Lambda\times\RR^+$. 

\begin{lemma}\label{l.key}(See \cite[Lemma 6.2]{PYY21})
	Let $(\cP_1,\cG_1,\cS_1)$ be a decomposition for $\cD_1\subset \cO(U_1)$ such that
	\begin{enumerate}
		\item $(\cG_1)^1$ has (tail) specification at scale $\delta$  with shadowing orbit contained in $U$;
		\item the potential $\phi$ has the Bowen property on $\cG_1$ at scale $\vep$; and
		\item  $P((\cO(U_1)\setminus\cD_1)\cup [\cP_1]\cup[\cS_1], \phi,\delta,\vep)<P(\phi;\Lambda)$.
	\end{enumerate}
	For every $\alpha_1>0$ and $\gamma\in I_\delta$, let $M = M(\gamma,\alpha_1), T_1 = T_1(\gamma,\alpha_1) $ be the constants given by~\eqref{e.MT}, i.e.,  from Lemma \ref{l.4.8} with $\alpha_2=\frac12$. Then there exists $L = L(\gamma,\alpha_1)$, such that the following statement holds:
	\begin{itemize}
		\item 
		for every $t\ge T_1$ and $ \cC\subset \Lambda\times\RR^+$ 
		with $\lambda(\cC,2\gamma,2\gamma,t)\ge\alpha_1e^{tP(\phi;\Lambda)}$, there exist $t_p,t_s\in[0,[M]]\cap \NN$, such that, with $T(t) = t-(t_p+t_s)$,
		\begin{equation}\label{e.key}
			\lambda\left(f_{t_p,t_s}(\cC)\cap (\cG_0)^1,{\gamma,3\gamma},T(t)\right)\ge L \cdot \lambda(\cC, 2\gamma,2\gamma,t),
		\end{equation}
		where $f_{t_p,t_s}(\cC)$ is defined by~\eqref{e.fC}.
	\end{itemize}
\end{lemma}
The proof is a verbatim copy of \cite[Lemma 6.2]{PYY21} using the pigeon hole principle and is omitted.

Similarly, using Lemma \ref{l.4.8'} we obtain a version of Lemma \ref{l.key} for $\cG_1$ with $\cC\subset \cO(U_1)$.
\begin{lemma}\label{l.key1}
	Let $(\cP_1,\cG_1,\cS_1)$ be a decomposition for $\cD_1\subset  \cO(U_1)$ such that
	\begin{enumerate}
		\item $(\cG_1)^1$ has (tail) specification at scale $\delta$  with shadowing orbit contained in $U$;
		\item the potential $\phi$ has the Bowen property on $\cG_1$ at scale $\vep$; and
		\item $P((\cO(U_1)\setminus\cD_1)\cup [\cP_1]\cup[\cS_1], \phi,\delta,\vep)<P(\phi;\Lambda)$.
	\end{enumerate}
	For every $\alpha_1>0$ and $\gamma\in I_\delta$, let $M = M(\gamma,\alpha_1), T_1 = T_1(\gamma,\alpha_1) $ be the constants given by~\eqref{e.MT}, i.e.,  from Lemma \ref{l.4.8} with $\alpha_2=\frac12$. Then there exists $L = L(\gamma,\alpha_1)$, such that the following statement holds:
	\begin{itemize}
		\item 
		for every $t\ge T_1$ and $ \cC\subset \cO(U_1)$ 
		with $\lambda(\cC,2\gamma,2\gamma,t)\ge\alpha_1e^{tP(\phi;\Lambda)}$, there exist $t_p,t_s\in[0,[M]]\cap \NN$, such that, with $T(t) = t-(t_p+t_s)$,
		\begin{equation}\label{e.key1}
			\lambda\left(f_{t_p,t_s}(\cC)\cap (\cG_1)^1,{\gamma,3\gamma},T(t)\right)\ge L \cdot \lambda(\cC, 2\gamma,2\gamma,t),
		\end{equation}
		where $f_{t_p,t_s}(\cC)$ is defined by~\eqref{e.fC}.
	\end{itemize}
\end{lemma}

Next we improve the lower bound on $(\cG_0)^M$ (Lemma \ref{l.4.10}) to a lower bound on $(\cG_0)^1$.

\begin{lemma}\label{l.cor1}(See \cite[Lemma 6.5]{PYY21})
	Let $\vep,\delta,\gamma$ be as in Lemma \ref{l.key}, and let $M,T_1$ be given by~\eqref{e.MT} with $\alpha_1=1$. Then there exists a  constant $L_2 = L_2(\gamma)\in\RR^+$ such that for every $t\ge T_1$, there exist $T(t)\in[t-2M,t]$ and $L_3\in\RR^+$ such that
	\begin{equation}\label{e.4.5.1}
		\lambda((\cG_0)^1,\gamma,3\gamma,T(t))\ge e^{-L_2} e^{T(t)P(\phi;\Lambda)},
	\end{equation}
	\begin{equation}\label{e.4.5.2}
		\lambda((\cG_0)^1,\gamma,T(t))\ge e^{-L_3} e^{T(t)P(\phi;\Lambda)}. 
	\end{equation}
\end{lemma}
\begin{proof}
	Take $\cC = \Lambda\times\RR^+$ with $\alpha_1 = 1$ in Lemma \ref{l.key}. 
 Then \begin{align*}
     \lambda((\cG_0)^1,{\gamma,3\gamma},T(t))\ge\,&\lambda(f_{t_p,t_s}(\Lambda\times\RR^+)\cap (\cG_0)^1,{\gamma,3\gamma},T(t))\\ \ge\,& L \cdot \lambda(\Lambda\times\RR^+, 2\gamma,2\gamma,t)\\ 
     \ge\,&L e^{tP(\phi;\Lambda)}\\
     \ge\,&L' e^{T(t)P(\phi;\Lambda)},
 \end{align*}
where the last inequality is due to $|T(t)-t|\le 2M$  (and $L'$ depends on $M$). Then, \eqref{e.4.5.2} is proved using the Bowen property.

\end{proof}

We conclude this section with the following upper and lower bounds of $\Lambda\times\RR^+$ and $\cO(U_1)$. 


\begin{lemma}\label{l.4.11}
	Assume that $(\cG_1)^1$ has tail specification at scale $\delta$  with shadowing orbit contained in $U$, $P_{\exp}^\perp (\phi,\vep)<P(\phi;\Lambda)$, and $\phi$ has the Bowen property  on $\cG_1$ at scale $\vep$.  For any $\gamma\in I_\delta$, let $T_1 = T_1(\gamma,1)$ be the constant given in Lemma \ref{l.key} with $\alpha_1 =1$. Then there exists  $C_3 = C_3(\gamma)>0$ such that for every $t\ge T_1$,
	$$
	C_3^{-1}e^{tP(\phi;\Lambda)}\le\lambda( \Lambda\times\RR^+, 2\gamma,t)\le \lambda(\cO(U_1),2\gamma,2\gamma, t)
	\le C_3 e^{tP(\phi;\Lambda)}.
	$$
	
\end{lemma}
\begin{proof}
	The first inequality is Lemma \ref{l.4.10}. The second one is by monotonicity on the second scale and the observation that $\Lambda\times\RR^+\subset \cO(U_1)$. Last one is from Lemma \ref{l.key1} with $\cC = \cO(U_1)$ and Proposition \ref{p.4.7}:
	$$ \lambda(\cO(U_1), 2\gamma, 2\gamma,t)\le L^{-1}\lambda((\cG_1)^1, \gamma, 3\gamma,T(t)) \le C_2L^{-1}e^{tP}.$$
\end{proof}
%
%

We will need the upper bound of $\cO(U_1)$ in the proof of the Gibbs property in the next section (see Lemma \ref{l.gibbs}). This is due to the need to use specification on $\Lambda\times\RR^+$ and the shadowing orbits are therefore contained in $\cO(U_1)$. We remark that in this section, only the specification on $(\cG_1)^1$ is needed (Assumption \hyperlink{I1}{$(I_1)$}). However, Assumption \hyperlink{I0}{$(I_0)$} will be useful in the next section when we prove the lower Gibbs bound on $\cG_0\subset \Lambda\times\RR^+$. 

\section{Gibbs Property}\label{sec4}
 Let  $\rho = 22\cdot \delta$ and $\rho_1 = 20\cdot \delta$;
we now construct an equilibrium state supported on $\Lambda$. {The construction is standard: for any $t>0$, we let $E_t$ be a $(t, \rho_1)$-separated set of $(\cO(U_1))_t$ with maximal cardinality, satisfying
\begin{equation}\label{e.es}
\sum_{x \in E_{t}} e^{\Phi_{0}(x, t)} \ge  \lambda( \cO(U_1), \rho_1, t) -1.
\end{equation}

}
Then, as in the proof of Lemma \ref{l.equalpressure}, define the measures
\begin{align*}\numberthis\label{e.nu}
&\nu_t = \frac{\sum_{x \in E_t} e^{\Phi_0(x, t)}\cdot \delta_x}{\sum_{x \in E_t} e^{\Phi_0(x, t)}},\\
&\mu_t = \frac{1}{t} \int_0^t (f_s)_{\ast}v_t ds.
\end{align*}
We take a subsequence $\mu_{t_k}$ converging to a probability measure $\mu$ which must be invariant. Since we assume $P_{\exp}^\perp (\phi,\vep;\Lambda)<P(\phi;\Lambda)$ (and recall that $\vep \ge 1000\delta>\rho_1$), 
the proof of the variational principle shows that $\mu$ is an equilibrium state. 
We now establish the (lower) Gibbs property\footnote{Note that the lower Gibbs property may not hold on $\cG_1$ or $(\cG_1)^1$ since $\mu$ is supported  on $\Lambda$, and orbits in $\cG_1$ may not be in $\Lambda$ (although their Bowen balls may intersect with $\Lambda$). } of $\mu$ on $(\cG_0)^1.$ This is the main ingredient for the uniqueness of the equilibrium state.
    
\begin{lemma}\label{l.gibbs}
	There exist $T_2>0$, $Q>0$ such that for every $(x,t)\in (\cG_0)^1$ with $t>T_2$, we have 
	\begin{equation}\label{e.gibbs}
		\mu(B_t(x,\rho))\ge Qe^{-tP(\phi;\Lambda)+\Phi_0(x,t)}.
	\end{equation}
\end{lemma}
\begin{proof}
Let $M, T_1$ be given by~\eqref{e.MT} with $\gamma=2\rho,\ \alpha_1=1.$ Let 
$$
T_2 = \max\{T_1, T_0, 2M+100\tau\}
$$ 
where $T_0, \tau$ are given by Definition \ref{d.spec11}. Next, we fix any $t_0\ge T_2$ and take, once and for all, an orbit segment $(x_0,t_0)\in (\cG_0)^1$. 
 
 For any $t>T_2,$
 we estimate $\nu_s(f_{-r}(B_t(x,\rho)))$ for any  $r\in (T_1+ \tau, s - t - 2\tau - 2T_1).$
	
 Let $u_1 = r-\tau$ and $u_2 = s-t-r-\tau$ (see \cite[Fig. 4]{CT16}),
 we first apply Lemma \ref{l.cor1} to get $(T(u_1), \rho)$-separated sets $E'_{T(u_1)}$ of $((\cG_0)^1)_{T(u_1)}$ and $(T(u_2), \rho)$-separated sets $E'_{T(u_2)}$ of $((\cG_0)^1)_{T(u_2)}$, satisfying
 \begin{align}
    \sum_{x \in E'_{T(u_1)}}e^{\Phi_{0}(x, T(u_1))} \geq  \frac{1}{2}e^{-L_3} e^{T(u_1)P(\phi;\Lambda)} \label{e.gibbs.s1}\\
    \sum_{x \in E'_{T(u_2)}}e^{\Phi_{0}(x, T(u_2))} \geq  \frac{1}{2}e^{-L_3} e^{T(u_2)P(\phi;\Lambda)}.\label{e.gibbs.s2}
 \end{align}
 
  Let $E_s$ be the $(s,\rho_1)$-separated set of $(\cO(U_1))_s$ with maximal cardinality in the construction of $\mu$.  By the specification on $(\cG_0)^1$ with shadowing orbit in $U_1$, given $\cx \in  E'_{T(u_1)}\times E'_{T(u_2)}$  we obtain times $\tau_i\le \tau$, $i=1,2,3,4,$ together with a point $y(\cx)\in (\cO(U_1))_{s_4+t_0}$ where  
  $$
  s_4 = t_0 + \tau_1 + T(u_1)+ \tau_2 + t + \tau_3 + T(u_2) + \tau_4,
  $$  
  satisfying
 \begin{align*}
     y(\cx) &\in B_{t_0}(x_0,  \delta) \\
     f_{s_1}y(\cx):=f_{t_0 + \tau_1}y(\cx) & \in B_{T(u_1)}(x_1, \delta) \\
     f_{s_2}y(\cx):=f_{t_0 + \tau_1 + T(u_1)+ \tau_2}y(\cx) & \in B_{t}(x, \delta) \\
     f_{s_3}y(\cx):=f_{t_0 + \tau_1 + T(u_1)+ \tau_2 + t + \tau_3}y(\cx) &\in B_{T(u_2)}(x_2,  \delta) \\
     f_{s_4}y(\cx):=f_{t_0 + \tau_1 + T(u_1)+ \tau_2 + t + \tau_3 + T(u_2) + \tau_4}y(\cx) &\in B_{t_0}(x_0, \delta).
 \end{align*}
 See Figure \ref{f.2}.
 
 \begin{figure}[h!]
 	\centering
 	\def\svgwidth{\columnwidth}
 	\includegraphics[scale=0.9]{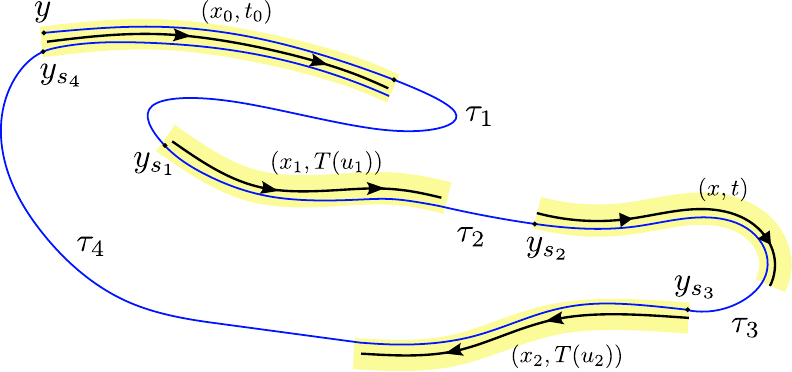}
 	\caption{The shadowing orbit $(y, s_4+t_0)\in \cO(U_1)$}
 	\label{f.2}
 \end{figure}
  
 To simplify notation, define $\iota= s_2 -r$. Then, the point $y_\iota = f_\iota(y)$ satisfies that $f_r(y_\iota) = y_{s_2}\in B_t(x,\delta).$ 
  Furthermore, we have 
 \begin{alignat*}{2}	
 	\iota =\,& t_0 + \tau_1 + T(u_1) + \tau_2 - r\\
 	\geq\,& t_0 + \tau_1 + (r - \tau -2M) + \tau_2 - r \qquad &&  (u_1 = r-\tau,\ T(u_1) \ge u_1 - 2M)  \\
 	=\,& t_0 + \tau_1  - \tau -2M + \tau_2  \\
 	\geq \,& T_2 - \tau -2M && (\tau_1,\tau_2\ge 0,\  t_0\ge T_2)\\
 	> \,& 0.
 \end{alignat*}
 More over, 
 \begin{alignat*}{2}	
	\iota + s =\,& s_2 - r + s\\
	= \,& s_2 + (s-t-r-\tau-2M)+t+\tau+2M \\
	=  \,& s_2 + (u_2-2M)+t+\tau+2M  \qquad &&  (u_2 = s-t-r-\tau)\\
	\le\, & s_2 +  T(u_2) +t +\tau+2M \qquad &&  (T(u_2) \ge u_2 - 2M)\\
	= \, &s_4 - \tau_3-\tau_4 + \tau + 2M\qquad && (s_4 = s_2 + t + \tau_3 + T(u_2) + \tau_4)\\
	\le \,& s_4 +T_2 \qquad && (T_2 \ge 2M+100\tau)\\
	\le \,& s_4+t_0.
 \end{alignat*}

Now we consider the orbit segment $(y_\iota,s)$. Since $y\in (\cO(U_1))_{s_4+t_0}$ (and recall that $(\cO(U_1))_{t_1}\subset (\cO(U_1))_{t_2}$ whenever $t_1 \ge t_2$), the previous inequalities show that 
\begin{equation}\label{e.wd}
y_\iota\in (\cO(U_1))_{s}.
\end{equation}

Since $E_s$ is a $(s,\rho_1)$-separated set of $(\cO(U_1))_{s}$ with maximal cardinality, it is also a $(s,\rho_1)$-spanning set of $(\cO(U_1))_{s}$. This allows us to define the map
$\pi: E'_{T(u_1)}\times E'_{T(u_2)} \to E_s$ that maps $y_\iota$ to the nearest point in $E_s$ under the $d_s$-metric; consequently,
$\pi(\cx)$ satisfies
\[d_s(\pi(\cx), y_\iota) \le \rho_1.\] 
Since $\rho = \rho_1 + 2\delta$, we obtain
\begin{align*}
     d_{t}(x, f_{r}\pi(\cx))\le &\, d_{t}(x, f_{r}y_\iota) + d_{t}(f_{r}y_\iota, f_{r}\pi(\cx)) \qquad\qquad\\
     \le &  \delta + \rho_1 < \rho,\\
     d_{T(u_1)}(x_1, f_{s_1-\iota}\pi(\cx))\le\, & d_{T(u_1)}(x_1, f_{s_1-\iota}y_\iota) + d_{T(u_1)}(f_{s_1-\iota}y_\iota, f_{s_1-\iota}\pi(\cx)) \\
     = \,& d_{T(u_1)}(x_1, f_{s_1}y) + d_{T(u_1)}(f_{s_1-\iota}y_\iota, f_{s_1-\iota}\pi(\cx)) \\
     \le \,&  \delta + \rho_1 < \rho,
 \end{align*} 
 and
 \begin{align*}
     d_{T(u_2)}(x_2, f_{s_3-\iota}\pi(\cx))\le \,& d_{T(u_2)}(x_2, f_{s_3-\iota}y_\iota) + d_{T(u_2)}(f_{s_3-\iota}y_\iota, f_{s_3-\iota}\pi(\cx)) \\
     = \,& d_{T(u_2)}(x_2, f_{s_3}y) + d_{T(u_2)}(f_{s_3-\iota}y_\iota, f_{s_3-\iota}\pi(\cx)) \\
     \le \,&  \delta + \rho_1 < \rho.
 \end{align*}
In the estimate above, we need orbit segments $(f_{s_1-\iota}\pi(\cx), T(u_1))$ and $(f_{s_3-\iota}\pi(\cx), T(u_2))$ to be contained in $(\pi(\cx), s);$  for this, we shall check 
 $(s_1-\iota,s_3-\iota+T(u_2))\subset (0,s)$.
This is proved by noting that  
\begin{align*}
    s_1 - \iota =\,& r - T(u_1) - \tau_2 \\
    =\,& r - \tau_2 - T(r - \tau) \\
    \geq\,& r - \tau - (r - \tau)\\
    = \,&0,
\end{align*} and
\begin{align*}    
    s_3-\iota+T(u_2) = \,& t + \tau_3 +r + T(u_2) \\
    = \,&t + \tau_3 +r + T(s-t-r-\tau) \\
    \le\,&  s.
\end{align*}

The same proof as in Lemma \ref{l.4.3} shows that $\pi$ has finite multiplicity $(C_1)^2$  where $C_1$ is the constant given by Lemma \ref{l.4.3}. 
Moreover, by the Bowen property on $(\cG_0)^1$ at scale $\vep \ge 1000 \delta$ we obtain, for some constant $C_4>0,$
 \begin{align}\label{e.6.1a}
        \Phi_{0}(\pi \mathbf{x}, s)\ge 
         \,& -C_4 +\Phi_{0}(x_1, T(u_1)) 
         + \Phi_{0}(x, t)+ \Phi_{0}(x_2, T(u_2)).
    \end{align}
    
It remains to estimate
	\begin{align*}	 
		&\nu_s(f_{-r}(B_t(x,\rho)))\\
		=\,&\frac{\sum_{z\in E_s} e^{\Phi_0(z,s)} \delta_z(f_{-r}(B_t(x,\rho)))}{\sum_{z\in E_s} e^{\Phi_0(z,s)}}\\
	 	\geq \,&  C_1^{-2}\left(\sum_{z\in E_s} e^{\Phi_0(z,s)}\right)^{-1}\cdot\sum_{\textbf{x}\in E_{T(u_1)}^\prime\times E_{T(u_2)}^\prime} e^{\Phi_0(\pi(\textbf{x}),s)}\\
		\ge \,&  C_5 e^{\Phi_0(x,t)}\left(\sum_{z\in E_s} e^{\Phi_0(z,s)}\right)^{-1}\sum_{x_1\in E_{T(u_1)}^\prime} e^{\Phi_0(x_1,T(u_1))} \\
        & \times \sum_{x_2\in E_{T(u_2)}^\prime} e^{\Phi_0(x_2,T(u_2))},
	\end{align*}
	where the last inequality follows from \eqref{e.6.1a}, and $C_5 = (C_1)^{-2}\cdot e^{-C_4}.$
	
	 For the term $\left(\sum_{z\in E_s} e^{\Phi_0(z,s)}\right)^{-1}$, we use the upper bound on $\cO(U_1)$ (Lemma \ref{l.4.11}); for the last two terms, we use \eqref{e.gibbs.s1} and \eqref{e.gibbs.s2}, namely the lower bound on $(\cG_0)^1$. This gives
	\begin{align*}
	\nu_s(f_{-r}(B_t(x,\rho))) & \ge C_6 \cdot e^{\Phi_0(x,t)}\cdot e^{-sP} \cdot e^{T(r - \tau)P}\cdot e^{(T(s-t-r-\tau))P} \\
	& \ge  C_6\cdot e^{-(2\tau+4M) P}\cdot e^{\Phi_0(x,t)}\cdot e^{-tP}\\
         &=  C_7 \cdot e^{\Phi_0(x,t)-tP}.
	\end{align*}

The constant $C_7$ above doesn't rely on $r$; this allows us to integrate over $r$ and get

\begin{align*}
	\mu_s(B_t(x, \rho)) &\geq \frac{1}{s} \int_{T_0+ \tau}^{s - t - 2\tau - 2T_0} \nu_s(f_{-r}(B_t(x,\rho))) dr\\
	&\geq ( 1 - \frac{t + 3\tau + 3T_0}{s} ) \cdot C_7\cdot  e^{\Phi_0(x,t)-tP}.
	\end{align*}
The conclusion of the lemma follows by letting $s \rightarrow \infty$.
 
\end{proof}

We remark that the use of $(x_0, t_0)$ is to enlarge the time interval $I$ of $y$ for which the orbit segment $f_I(y)$ is contained in $\cO(U_1)$. Because $T(u_1), T(u_2)$ are shorter than $u_1, u_2$ by at most $2M$ respectively, we need the time interval  $I$ to have  length at least \( s + 4M \) to guarantee that \( \pi: E'_{T(u_1)}\times E'_{T(u_2)} \to E_s \) is well-defined (see the argument that leads  to \eqref{e.wd}). 

We also obtain the following lemma which was used in \cite{CT16, PYY23} to establish the ergodicity of $\mu$.
\begin{lemma}\label{l.mixgibbs}
	There exist $T_2>0$, $Q>0$ such that for every $(x_1,t_1), (x_2,t_2)\in (\cG_0)^1$ with $t_1, t_2>T_2$, and every $q > T_2 + 2\tau$, there exists $q' \in [q - 2\tau - 2M, q]$ such that
 $$\mu(B_{t_1}(x_1,\rho) \cap f_{t_1+q'}B_{t_2}(x_2,\rho))\ge Q'e^{-(t_1+t_2)P(\phi;\Lambda)+\Phi_0(x_1,t_1)+\Phi_0(x_2,t_2)}.$$
 Furthermore, we can choose $N \in \cN$ such that $q'$ can be taken such that $q' = q - k - \frac{2i}{N} \tau$ for some $k \in\{0, 1, ...,2[M]\}$ and $i \in\{0, 1, ...,N\}$.
\end{lemma}

The proof follows the modification as in the proof of Lemma \ref{l.gibbs} applied to the proof of \cite[Lemma 7.3]{PYY21}, which itself is a modification of \cite[Lemma 4.16, Lemma 4.17]{CT16}. An orbit segment $(x_0,t_0)\in (\cG_0)^1$ is fixed once and for all. Then, for any $\cx = (y_1,y_2,y_3)\in E^1\times E^2\times E^3$ where $(y_1,s_1), (y_2,s_2), (y_3,s_3)$ are taken from some separated sets $E^1,E^2,E^3$ of $(\cG_0)^1$ with appropriate length, we choose  an shadowing orbit $(y,T) \in\cO(U_1)$ such that  it shadows orbit segments at scale $\delta$ in the following order: 
$$
(x_0,t_0)\to (y_1,s_1)\to (x_1,t_1)\to (y_2,s_2)\to (x_2,t_2)\to (y_3,s_3)\to (x_0,t_0).
$$
Let $E_s$ be the $(s,\rho_1)$-separated set of $(\cO(U_1))_s$ as before. The choice of $t_0\ge T_2\ge  2M+100\tau$ guarantees that the map $\pi: E^1\times E^2\times E^3\to E_s$ is well-defined and have bounded multiplicity. Then, we use Lemma \ref{l.4.11} to obtain the upper bound on the partition sum of $E_s\subset U_1$ as we did in Lemma \ref{l.gibbs}, and use Lemma \ref{l.cor1} for the lower bound on the partition sum of $E^1,E^2,E^3\subset  \Lambda\times \RR^+$, 
 and the rest of the proof follows. The details are omitted.

\section{Finishing the proof of Theorem \ref{m.1}}\label{sec5}

At this point, we have re-established the main tools to prove the ergodicity of $\mu$ (Lemma \ref{l.mixgibbs} replacing \cite[Lemma 7.3]{PYY21}) and to prove the uniqueness of equilibrium state (Lemma \ref{l.gibbs} replacing \cite[Lemma 7.1]{PYY21}). The rest of the proof of Theorem \ref{m.1} does not use any specifications (other than those that have been used in the proof of previous lemmas), and therefore is a verbatim copy of \cite[Section 8]{PYY21} with the space $\bM$ replaced by the invariant set $\Lambda$. For this reason, below we will only highlight the necessary steps of the proof and omit all the details.  
%
\begin{lemma}(See \cite[Lemma 4.18]{CT16},\cite[Lemma 8.1]{PYY21})\label{l.4.18}
Let $\epsilon, \delta $  be as before and let $\gamma \in I_{\delta}$. For every $\alpha \in (0, 1)$, there exists a constant $C_{\alpha} > 0$ with the following property: let $\nu$ be any equilibrium state for the potential $\phi$ whose support is contained in $\Lambda$, and let $\{E_t\}_{t>0}$ be a family  $(t, 2\gamma)$-separated sets with maximal cardinality for $\lambda(\Lambda, 2\gamma, t)$ with adapted partitions $A_t$. Then let $T_1 = T_1(\gamma,1)$ be the constant given in Lemma \ref{l.key} with $\alpha_1 =1$, for every $t > T_1$, if $ E_t' \subset E_t$ satisfies $\nu (\bigcap_{x\in E_t'} w_x ) \geq \alpha$, then letting $C = \{(x, t): x \in E_t'\}$, we have$$\lambda(C, 2\gamma, 2\gamma, t) \geq C_\alpha e^{tP(\phi;\Lambda)}.$$
\end{lemma}

\begin{lemma}\label{p.8.3}(See \cite[Lemma 8.3]{PYY21})
Let $ \nu_1, \nu_2$ be two invariant probability measures supported in $\Lambda$ that are almost expansive on $\Lambda$ at scale $\epsilon$ and mutually singular. Then for $\epsilon_1 < \epsilon$, every $\beta > 0$, and every $M > 0$, there exist compact sets $Q_t$ contained in $\Lambda$ for all t sufficiently large, satisfying $\nu_1(Q_t) \geq 1 - \beta $ and $\nu_2(Q_t) = 0$, such that for every $0 \le i, j \le M$, we have
$$\limsup_{t \rightarrow \infty} \nu_2 (\bigcup_{y \in Q_t}f_{-i}(B_{t-(i+j)}(f_i(y), \epsilon_1))) \le \beta.$$
\end{lemma}
\begin{lemma}\label{l.nosingular}(See \cite[Propotition 8.2]{PYY21})
Assume that the assumptions of theorem \ref{m.1} hold. Then no equilibrium state $\nu$ is mutually singular with $\mu$. 
\end{lemma}

And it remains to show that $\mu$ is ergodic.
\begin{lemma}\label{l.ergodic} (See \cite[Propotition 8.4]{PYY21})
The equilibrium state $\mu$ constructed is ergodic.
\end{lemma}

Then, Theorem \ref{m.1} follows immediately from Lemma \ref{l.nosingular} and \ref{l.ergodic}.

We conclude this article with the following upper Gibbs property of $\mu$. This result is not used in this article, but may be of independent interest in future projects. 

\begin{proposition}\label{p.uppergibbs}
	Assume that the assumptions of Theorem \ref{m.1} hold and let $\gamma=10\delta = \frac12\rho_1\in I_\delta$. Then, there exists $Q>0$ such that for every $(x,t)\in \Lambda\times\RR^+$, the unique equilibrium state $\mu$ satisfies
	$$
	\mu(B_t(x,\gamma))\le Qe^{-tP(\phi; \Lambda) + \Phi_\gamma(x,t)}.
	$$
	Furthermore, there exists $Q'>0$ such that for every $(x,t)\in (\cG_0)^1$ one has  
	$$
	\mu(B_t(x,\gamma))\le Q'e^{-tP(\phi; \Lambda) + \Phi_0(x,t)}.
	$$
\end{proposition}

\begin{proof}
	The proof is similar to that of \cite[Proposition 4.21]{CT16}. 
	
	For each $s>0$ let $E_s$ be a $(s,\rho_1)$-separated set of $\Lambda$ that achieves $\lambda(\Lambda, \rho_1,s)$, and $\nu_s,\mu_s$ be defined as in \eqref{e.nu}. The same argument as in Section \ref{sec4} shows that as $t\to\infty$, $\mu_t$ must converge to $\mu$, the unique equilibrium state on $\Lambda$.  Below, we will estimate  $\nu_s(f_{-r}(B_t(x,\gamma)))$ for any  $r\in (0,s-t)$ and sufficiently large $s$. 
	
	For this purpose, for each $r>0$ we let $E_r'$ be a $(r,\gamma)$-separated set of $\Lambda$ that achieves $\lambda(\Lambda, \gamma, r)$, and consider the map $\pi: E_s\cap f_{-r}(B_t(x,\gamma))\to E'_{r}\times E'_{s-t-r}$ defined by $d_r(\pi(x)_1, x)<\gamma$ and $d_{s-t-r}(\pi(x)_2,f_{r+t}(x))<\gamma$ (here we write $\pi(x) = (\pi(x)_1, \pi(x)_2)$). The map is well-defined since each $E'_r$ is spanning on $\Lambda$; it is also easy to check (see \cite[Lemma 4.1]{CT16}) that $\pi$ is injective. By Lemma \ref{l.4.11}  
	we have, for all $s>0$ (increase $C_3$ if necessary)
	$$
	C_3^{-1}e^{sP(\phi;\Lambda)}\le\lambda( \Lambda, 2\gamma,s)
	\le C_3 e^{sP(\phi;\Lambda)},
	$$
	and, by monotonicity, 
	$$
	\lambda(\Lambda,\gamma,\gamma, s)
	\le	\lambda(\cO(U_1),\gamma,\gamma, s)
	\le C_3 e^{sP(\phi;\Lambda)}.
	$$
	Then, we estimate
	\begin{align*}	 
		&\nu_s(f_{-r}(B_t(x,\rho)))\\
		=\,&\frac{\sum_{z\in E_s} e^{\Phi_0(z,s)} \delta_z(f_{-r}(B_t(x,\rho)))}{\sum_{z\in E_s} e^{\Phi_0(z,s)}}\\
		\le \,& \big(\lambda(\Lambda,2\lambda,s)\big)^{-1}\sum_{y\in E'_r}\sum_{z\in E'_{s-t-r}} e^{\Phi_\gamma(y,r)}\,e^{\Phi_\gamma(x,t)}\,e^{\Phi_\gamma(z,s-t-r)}\\
		\le \,& C_3 e^{-s P(\phi;\Lambda)}\cdot  \lambda(\Lambda,\gamma,\gamma, r)\cdot \lambda(\Lambda,\gamma,\gamma, s-t-r) \cdot e^{\Phi_\gamma(x,t)}\\
		\le \,&  (C_3)^3   e^{-s P(\phi;\Lambda)} \,  e^{r P(\phi;\Lambda)} \,  e^{s-t-r P(\phi;\Lambda)}\,e^{\Phi_\gamma(x,t)} \\
		= \,& (C_3)^3  e^{-tP(\phi;\Lambda) + \Phi_\gamma(x,t)}.
 	\end{align*} 
	The second statement of the lemma follows from the Bowen property on $\cG_0$. 
\end{proof}
\begin{remark}
Indeed the same result holds for $(x,t)\in \cO(U_1)$ since Lemma \ref{l.4.11} allows us to control $\lambda(\cO(U_1),2\lambda,t)$ and $\lambda(\cO(U_1),\lambda,\lambda,t)$ from both above and below; however, for those orbit segments $\mu(B_t(x,\gamma))$ may be zero, and the statement becomes trivial. 
\end{remark}

\end{document}